\documentclass[12pt]{amsart}
\usepackage{color}
\usepackage{amssymb}
\usepackage{comment}
\usepackage{pdfpages}
\usepackage{graphicx}
\usepackage{dcolumn}

\RequirePackage[numbers]{natbib}
\RequirePackage{hyperref}
\usepackage{paralist}

\usepackage{tikz} 
\usepackage{dcolumn}
\usetikzlibrary{arrows,positioning}
\tikzset{
    >=stealth',
    pil/.style={
           ->,
           thick,
           shorten <=2pt,
           shorten >=2pt,}
}

\usepackage{graphicx}
\graphicspath{{./figures/}}
\usepackage{amsfonts}
\usepackage{amsmath}
\usepackage{amsthm}
\usepackage{amssymb}
\usepackage{amsbsy}
\usepackage{epsfig}
\usepackage{fullpage}
\usepackage{natbib, mathrsfs} 
\usepackage{verbatim}
\usepackage[latin1]{inputenc}
\usepackage{mhequ}
\usepackage{algorithm}
\usepackage{algorithmic}
\usepackage{bbm}

\numberwithin{equation}{section}

\def \be{\begin{equs}}
\def \ee{\end{equs}}
\def \P{\mathbb{P}}
\def \E{\mathbb{E}}

\def \gd{G}

\def \bd{B}
\def \cov{W}

\newtheorem{theorem}{Theorem}[section]
\newtheorem{lemma}[theorem]{Lemma}

\newtheorem{assumptions}[theorem]{Assumptions}

\theoremstyle{plain}
\newtheorem{thm}{Theorem}
\newtheorem*{thm-non}{Theorem}

\theoremstyle{definition}
\newtheorem{defn}[theorem]{Definition}
\newtheorem{remark}[theorem]{Remark}

\begin{document}

\title[Metastability Bounds]{Simple Conditions for Metastability of Continuous Markov Chains}


\author{Oren Mangoubi$^{\flat}$}
\thanks{$^{\flat}$oren.mangoubi@gmail.com, 
\'{E}cole Polytechnique F\'{e}d\'{e}rale de Lausanne (EPFL), 
IC IINFCOM THL3,
1015 Lausanne, Switzerland}

\author{Natesh S. Pillai$^{\ddag}$}
\thanks{$^{\ddag}$pillai@fas.harvard.edu, 
   Department of Statistics,
    Harvard University, 1 Oxford Street, Cambridge
    MA 02138, USA}

\author{Aaron Smith$^{\sharp}$}
\thanks{$^{\sharp}$smith.aaron.matthew@gmail.com, 
   Department of Mathematics and Statistics,
University of Ottawa, 585 King Edward Avenue, Ottawa
ON K1N 7N5, Canada}

 \thanks{NSP is partially supported by an ONR grant. AS and OM were supported by a grant from NSERC}

\maketitle







\begin{abstract}
A family $\{Q_{\beta}\}_{\beta \geq 0}$ of Markov chains is said to exhibit \textit{metastable mixing} with \textit{modes} $S_{\beta}^{(1)},\ldots,S_{\beta}^{(k)}$ if its spectral gap (or some other mixing property) is very close to the worst conductance $\min(\Phi_{\beta}(S_{\beta}^{(1)}), \ldots, \Phi_{\beta}(S_{\beta}^{(k)}))$ of its modes. We give simple sufficient conditions for a family of Markov chains to exhibit metastability in this sense, and verify that these conditions hold for a prototypical Metropolis-Hastings chain targeting a mixture distribution.  Our work differs from existing work on metastability in that, for the class of examples we are interested in, it gives an asymptotically exact  formula for the spectral gap (rather than a bound that can be very far from sharp) while at the same time giving technical conditions that are easier to verify for many statistical examples. Our bounds from this paper are used in a companion paper \cite{mangoubi2018Cheeger} to compare the mixing times of the Hamiltonian Monte Carlo algorithm and a random walk algorithm for multimodal target distributions.
\end{abstract}

\section{Introduction}
It is well known that Markov chains targeting multimodal distributions, such as those that appear in mixture models, will often mix very slowly. Of course, some algorithms are still faster than others, and the present paper is motivated by the problem of comparing different MCMC (Markov Chain Monte Carlo) algorithms in this ``highly multimodal" regime. In this paper, we give a first step in this direction by finding some simple sufficient conditions under which we can find an explicit formula for the spectral gap for MCMC algorithms on multimodal target distributions. To be slightly more precise, we consider a sequence of Markov transition kernels $\{Q_{\beta}\}_{\beta \geq 0}$ with state space $\Omega$ partitioned into pieces $\Omega = \sqcup_{i=1}^{k} S_{\beta}^{(i)}$. One of our main results, Lemma \ref{LemmaMeta2}, gives sufficient conditions under which the spectral gap $\lambda_{\beta}$ of $Q_{\beta}$ is asymptotically given by the worst-case conductance $\Phi_{\min}(\beta) = \min(\Phi_{\beta}(S_{\beta}^{(1)}), \ldots, \Phi_{\beta}(S_{\beta}^{(k)}))$, in the sense:
\be \label{EqProtConc}
\lim_{\beta \rightarrow \infty} \frac{\log(\lambda_{\beta})}{\log(\Phi_{\min}(\beta))} = 1.
\ee 

Our work is closely related to two large pieces of the Markov chain literature: \textit{decomposition bounds} (see \textit{e.g.,}  \cite{woodard2009sufficient, woodard2009conditions,madras2002markov,jerrum2004elementary,pillai2017elementary}) and \textit{metastability bounds} (see \textit{e.g.,} the popular book \cite{olivieri2005large} and the references within the recent articles \cite{beltran2015martingale,landim2018metastable}). Our work differs from existing work in that, for the class of examples we are interested in, it gives an asymptotically exact  formula for the spectral gap (rather than a bound that can be very far from sharp) while at the same time giving technical conditions that are easier to verify for many statistical examples. In particular, we give simple sufficient conditions that work on continuous state spaces). We do not assume that our state space is discrete or compact, that we have precise knowledge of the boundary of the modes, or that we have precise knowledge of the ``typical" trajectories between modes. We believe that our ``in-between" results are a useful compromise that focuses on the most relevant properties for comparison of Markov chains targeting multimodal distributions that arise in a statistical context.

The main heuristic behind our calculations is that, in the highly-multimodal regime, a Markov chain with strongly multimodal stationary distribution will mix \textit{within} its starting mode before travelling between modes. When this occurs, we say that the Markov chain exhibits \textit{metastable behaviour},  and the mixing properties of the Markov chain are often determined by the rate of transition between modes at stationarity (again, see \cite{landim2018metastable} and the references therein). As a prototypical example, we consider the simple mixture of two Gaussians
\be 
\pi_{\sigma} = \frac{1}{2} \mathcal{N}(-1,\sigma^{2}) + \frac{1}{2} \mathcal{N}(1,\sigma^{2})
\ee 
for $\sigma > 0$. When $\sigma$ is close to 0, the usual tuning heuristic for the random walk Metropolis-Hastings (MH) algorithms (see \textit{e.g.,} \cite{roberts1997weak}) suggests using a proposal distribution with standard deviation on the order of $\sigma$, such as:
\be 
K_{\sigma}(x,\cdot) = \mathcal{N}(x,\sigma^{2}).
\ee 
Informally, an MH chain $\{X_{t}\}_{t \geq 0}$ with proposal distribution $K_{\sigma}$, target distribution $\pi_{\sigma}$ and initial point $X_{0} \in [-2,-0.5]$ in one of the modes will evolve according to the following three stages:
\begin{enumerate}
\item For $t$ very small, the law of the chain $X_t$, $\mathcal{L}(X_{t})$, will depend quite strongly on the starting point $X_{0}$.
\item For $\sigma^{-1} \ll t \ll e^{c_{1} \sigma^{-2}}$ and $c_{1} > 0$ small, the chain will have mixed very well on its first mode and is very unlikely to have ever left the interval $(-\infty, -0.1)$, so that:
\be 
\| \mathcal{L}(X_{t}) - \mathcal{N}(-1,\sigma^{2}) \|_{\mathrm{TV}} \ll e^{-c_{2} \sigma^{-1}}
\ee 
for some $c_{2} > 0$.
\item For $t \gg e^{c_{3} \sigma^{-2}}$, the chain will have mixed well on the entire state space in the sense that
\be 
\| \mathcal{L}(X_{t}) - \pi_{\sigma} \|_{\mathrm{TV}} \ll e^{-c_{3} \sigma^{-1}}
\ee 
for some $c_{3} > 0$.
\end{enumerate}

In the context of this example, the main result of our work is a straightforward way to verify that there is a sharp transition around $t \approx e^{\frac{1}{2} \sigma^{-2}}$, so that we may take $c_{1} = c_{3} = \frac{1}{2}$ in this heuristic description (see Theorems \ref{ThmRwmMultimodal1}, \ref{ThmRwmMultimodal2} for a precise statement). In the notation of Equality \eqref{EqProtConc}, we can take the parameter $\beta$  that indexes our chains  to be equal to $\sigma^{-1}$. We view $\beta$ as indexing ``how multimodal" a chain is, while in this particular example $\sigma^{-1}$ measures both the width of each mode \textit{and} how well-separated they are.

We believe that these scaling exponents $c_{1},c_{3}$ are natural ways to measure performance in the highly-multimodal regime; see our companion paper \cite{mangoubi2018Cheeger} for further discussion of this point and relationships to the literature on optimal scaling and lifted Markov chains.

\subsection{Guide to Paper}

In Section \ref{SecAlgDefs}, we review basic notation and definitions, and also provide some simple bounds. Our main results on metastability are in Section \ref{SecGenMetaBd}. Finally, we give an illustrative application in Section \ref{SecAppl}.

\section{Preliminaries} \label{SecAlgDefs}

\subsection{Basic Notation}
Throughout the remainder of the paper, we denote by $\pi$ the smooth density function of a probability distribution on a convex subset of $\mathbb{R}^{d}$. We denote by $\mathcal{L}(X)$ the distribution of a random variable $X$. Similarly, if $\mu$ is a probability measure, we write ``$X \sim \mu$" for ``$X$ has distribution $\mu$."  Throughout, we will generically let $Q \sim \pi$ and $P \sim \mathcal{N}(0, \mathrm{Id})$ be independent random variables, where $\mathrm{Id}$ is the $d$-dimensional identity matrix.

 For two nonnegative functions or sequences $f,g$, we write $f = O(g)$ as shorthand for the statement: there exist constants $0 < C_{1},C_{2} < \infty$ so that for all $x > C_{1}$, we have $f(x) \leq C_{2} \, g(x)$. We write $f = \Omega(g)$ for $g = O(f)$, and we write $f = \Theta(g)$ if both $f= O(g)$ and $g=O(f)$. Relatedly, we write $f = o(g)$ as shorthand for the statement: $\lim_{x \rightarrow \infty} \frac{f(x)}{g(x)} = 0$.  Finally, we write $f = \tilde{O}(g)$ if there exist constants $0 < C_{1},C_{2}, C_{3} < \infty$ so that for all $x > C_{1}$, we have $f(x) \leq C_{2} \, g(x) \log(x)^{C_{3}}$, and write $f = \tilde{\Omega}(g)$ for $g = \tilde{O}(f)$. As shorthand, we say that a function $f$ is ``bounded by a polynomial" if there exists a polynomial $g$ such that $f = O(g)$.

\subsection{Cheeger's inequality and the spectral gap}

We recall the basic definitions used to measure the efficiency of MCMC algorithms. Let $L$ be a reversible transition kernel with unique stationary distribution $\mu$ on $\mathbb{R}^{d}$. It is common to view $L$ as an operator from $L_{2}(\pi)$ to itself via the following formula:
\be
(L  f)(x) = \int_{y \in \mathbb{R}^{d}} L(x,dy) f(y).
\ee
The constant function is always an eigenfuncton of this operator, with eigenvalue 1. We define the space $W^{\perp} = \{ f \in L_{2}(\mu) \, : \, \int_{x} f(x) \mu(dx) = 0\}$ of functions that are orthogonal to the constant function, and denote by $L^{\perp}$ the restriction of the operator $L$ to the space $W^{\perp}$. We then define the \textit{spectral gap} $\rho$ of $L$ by the formula 
\be
\rho = \rho(L) \equiv 1 - \sup \{ |\lambda| \, : \, \lambda \in \mathrm{Spectrum}(L^{\perp}) \},
\ee
where $\mathrm{Spectrum}$ refers to the usual spectrum of an operator. If $L^{\perp}$ has a largest eigenvalue $|\lambda|$ (for example, if $L$ is a matrix of a finite state space Markov Chain), then $\rho = 1-| \lambda |$.  

Cheeger's inequality \cite{cheeger1970lower,lawler1988bounds} provides bounds for the spectral gap in terms of the ability of $L$ to move from any set to its complement in a single step. This ability is measured by the conductance $\Phi(L)$, which is defined by the pair of equations
\be 
\Phi(L) = \inf_{S  \in \mathcal{A} \, : \, 0 < \mu(S) < \frac{1}{2}} \Phi(L,S)
\ee
and 
\be 
\Phi(L,S) = \frac{ \int_{x} \mathbbm{1}\{x \in S\} L(x,S^{c}) \mu(dx)}{\mu(S) },
\ee 
where $\mathcal{A}= \mathcal{A}(\mathbb{R}^{d})$ denote the usual collection of Lebesgue-measurable subsets of $\mathbb{R}^{d}$.
Cheeger's inequality for reversible Markov chains, first proved in \cite{lawler1988bounds}, gives:
\begin{equation} \label{IneqCheegPoin}
\frac{\Phi(L)^2}{2} \leq \rho(L) \leq 2 \Phi(L).
\end{equation}

\subsection{Traces and Hitting Times}

We recall some standard definitions related to Markov processes.

\begin{defn} [Trace Chain]
Let  $K$ be the transition kernel of an ergodic Markov chain on state space $\Omega$ with stationary measure $\mu$, and let $S \subset \Omega$ be a subset with $\mu(S) > 0$. Let $\{X_{t}\}_{t \geq 0}$ be a Markov chain evolving according to $K$, and iteratively define
\be 
c_{0} &= \inf \{t \geq 0 \, : \, X_{t} \in S \} \\
c_{i+1} &= \inf \{t > c_{i} \, : \, X_{t} \in S \}.
\ee 
Then 
\be \label{EqTraceCoup}
\hat{X}_{t} = X_{c_{t}}, \quad t \geq 0
\ee 
is the \textit{trace of $\{X_{t}\}_{t \geq 0}$ on $S$}. Note that $\{\hat{X}_{t}\}_{t \geq 0}$ is a Markov chain with state space $S$, and so this procedure also defines a transition kernel with state space $S$. We call this kernel the \textit{trace of the kernel $K$ on $S$}.
\end{defn}

\begin{defn} [Hitting Time]
Let $\{X_{t}\}_{t \geq 0}$ be a Markov chain with initial point $X_{0} = x$ and let $S$ be a measurable set. Then 
\be 
\tau_{x,S} = \inf \{t \geq 0 \, : \, X_{t} \in S\}
\ee 
is called the \textit{hitting time} of $S$.
\end{defn}

\section{Generic Metastability Bounds} \label{SecGenMetaBd}

Denote by $\{Q_{\beta}\}_{\beta \geq 0}$ the transition kernels of ergodic Markov chains with stationary measures $\{ \mu_{\beta}\}_{ \beta \geq 0}$ on common state space $\Omega$, which we take to be a convex subset of $\mathbb{R}^d$. Throughout, we will always use the subscript $\beta$ to indicate which chain is being used - for example, $\Phi_{\beta}(S)$ is the conductance of the set $S$ with respect to the chain $Q_{\beta}$. For any set $S$ with $\pi_{\beta}(S) > 0$, define the restriction $\pi_{\beta} |_{S}$ of $\pi_{\beta}$ to $S$
\be 
\pi_{\beta} |_{S}(A) = \frac{\pi_{\beta}(S \cap A)}{\pi_{\beta}(S)}.
\ee 

Our two main results are:

\begin{enumerate}
\item In Lemma \ref{LemmaMeta1}, we fix a set $S \subset \Omega$ and give sufficient conditions for the \textit{worst-case} hitting time of $S^{c}$ from $S$ to be bounded by the \textit{average-case} hitting time $\Phi_{\beta}(S)$.
\item In Lemma \ref{LemmaMeta2}, we consider sufficient conditions on the entire partition $S^{(1)},\ldots,S^{(k)}$ to ensure that the \textit{spectral gap} of $Q_{\beta}$ is approximately equal to the \textit{worst-case conductance} $\min_{1 \leq i \leq k} \Phi_{\beta}(S^{(i)})$. 
\end{enumerate}

\subsection{Metastability and Hitting Times}

The main point of our first set of assumptions is to guarantee that the Markov chain cannot get ``stuck" for a long time before mixing within a mode $S$. Fix $S \subset \Omega$ with $\inf_{\beta \geq 0} \pi_{\beta}(S) \equiv c_{1} > 0$ and, for all $\beta \geq 0$, let $\gd_{\beta}, \bd_{\beta}, \cov_{\beta} \subset S$ satisfy 
\be 
\gd_{\beta} \subset \cov_{\beta}, \qquad \bd_{\beta} \subset \cov_{\beta}^{c}.
\ee 
In the following assumption, we think of the set $\gd_{\beta}$ as the points that are ``deep within" the mode $S$, the points $\bd_{\beta}$ as the points that are ``far in the tails" of the target distribution, and the ``covering set" $\cov_{\beta}$ as a way of separating these two regions.

\begin{assumptions} \label{AssumptionsMeta1}
 We assume the following all hold for $\beta > \beta_{0}$ sufficiently large:
\begin{enumerate}
\item \textbf{Small Conductance:} There exists some $c > 0$ such that $\Phi_{\beta}(S) \leq e^{-c \beta}$.
\item \textbf{Rapid Mixing Within $\gd_{\beta}$:}  Let $\hat{Q}_{\beta}$ be the Metropolis-Hastings chain with proposal kernel $Q_{\beta}$ and target distribution $\pi_{\beta} |_{S}$. There exists some function $r_{1}$ bounded by a polynomial such that 
\be  \label{IneqMeta1MixingRestrictions}
\sup_{x \in \gd_{\beta}} \| \hat{Q}_{\beta}^{r_{1}(\beta)}(x,\cdot) - \pi_{\beta} |_{S}(\cdot) \|_{\mathrm{TV}} \leq \beta^{-2} \Phi_{\beta}(S).
\ee 
\item \textbf{Never Stuck In $\cov_{\beta}\backslash \gd_{\beta}$:} There exists some function $r_{2}$ bounded by a polynomial such that 
\be \label{IneqNeverStuckInEscM1}
\sup_{x \in \cov_{\beta}\backslash \gd_{\beta}} \P[\tau_{x, \gd_{\beta} \cup S^{c}} > r_{2}(\beta)] \leq  \beta^{-2} \Phi_{\beta}(S).
\ee 
\item \textbf{Never Hitting $\cov_{\beta}^{c}$:} We have 
\be \label{IneqNeverHittingBad1}
\sup_{x \in \gd_{\beta} } \P[\tau_{x, \cov_{\beta}^{c}} < \min( r_{1}(\beta) + r_{2}(\beta) + 1, \tau_{x,S^{c}})] \leq \Phi_{\beta}(S)^{4}.
\ee 
\end{enumerate}
\end{assumptions}

Under these assumptions, we have the conclusion:

\begin{lemma} [Hitting Times and Conductance] \label{LemmaMeta1}
Let Assumptions \ref{AssumptionsMeta1} hold, and fix a point $x$ that is in  $\gd_{\beta} $ for all $\beta > \beta_{0}(x)$ sufficiently large. Then for all $\epsilon > 0$,
\be 
\P \left[ \frac{\log(\tau_{x,S^{c}})}{\log(\Phi_{\beta}(S))} > 1 + \epsilon \right] = o(1).
\ee 

\end{lemma}

\begin{proof}
Let $p \in \cov_{\beta}$, let $\{X_{t}\}_{t \geq 0}$ be a Markov chain with transition kernel $Q_{\beta}$ started at $X_{0} = p$, and let $\{Y_{t}\}_{t \geq 0}$ be a Markov chain with transition kernel $Q_{\beta}$ started with distribution $Y_{0} \sim \pi_{\beta}|_{S}$. Finally, define $T = T(\beta) = r_{1}(\beta) + r_{2}(\beta) + 1$.

We then calculate: 

\be 
\P[\tau_{p,S^{c}} \leq T ] &\geq \inf_{q \in \gd_{\beta}} \P[\tau_{q,S^{c}} \leq r_{1}(\beta) + 1] -  \P[X_{r_{2}(\beta)} \notin \gd_{\beta} \cup S^{c}] \\
&\geq \inf_{q \in \gd_{\beta}} \P[\tau_{q,S^{c}} \leq r_{1}(\beta) + 1] -  \P[\tau_{p, \gd_{\beta} \cup S^{c}} > r_{2}(\beta)] \\
&\geq \P[Y_{1} \in S^{c}] - \sup_{q \in \gd_{\beta}} \| \hat{Q}_{\beta}^{r_{1}(\beta)}(q,\cdot) - \pi_{\beta} |_{S}(\cdot) \|_{\mathrm{TV}} -   \P[\tau_{p, \gd_{\beta} \cup S^{c}} > r_{2}(\beta)] \\
&= \Phi_{\beta}(S) - \sup_{q \in \gd_{\beta}} \| \hat{Q}_{\beta}^{r_{1}(\beta)}(q,\cdot) - \pi_{\beta} |_{S}(\cdot) \|_{\mathrm{TV}} -    \P[\tau_{p, \gd_{\beta} \cup S^{c}} > r_{2}(\beta)]
\ee 

Applying parts \textbf{(2)} and \textbf{(3)} of Assumption \ref{AssumptionsMeta1} to bound the size of the negative terms, this implies 

\be \label{IneqMeta1Hit1}
\P[\tau_{p,S^{c}} \leq T]  \geq \Phi_{\beta}(S) (1 - 2 \beta^{-2}).
\ee 

If we also have $p \in \gd_{\beta} \subset \cov_{\beta}$, then applying part \textbf{(4)} of Assumption \ref{AssumptionsMeta1} gives 
\be \label{IneqMeta1Hit2}
\P[\tau_{p,\cov_{\beta}^{c}} \leq \min(T, \tau_{p,S^{c}})] \leq \Phi_{\beta}(S)^{4}. 
\ee 

We now iteratively apply Inequalities \eqref{IneqMeta1Hit1} and \eqref{IneqMeta1Hit2} to control the behaviour of $\{X_{t}\}_{t \geq 0}$ over longer time intervals. More precisely, for all $k \in \mathbb{N}$  and starting points  $p \in  \gd_{\beta}$, we have: 

\be 
\P[\tau_{p,S^{c}} > kT] &= \P[\tau_{p,S^{c}} > kT | \tau_{p,S^{c}} > (k-1)T, \, X_{(k-1)T} \in \cov_{\beta}] \P[ \tau_{p,S^{c}} > (k-1)T, \, X_{(k-1)T} \in \cov_{\beta} ] \\
& \qquad + \P[\tau_{p,S^{c}} > kT | \tau_{p,S^{c}} > (k-1)T, \, X_{(k-1)T} \in \cov_{\beta}^{c} ] \P[\tau_{p,S^{c}} > (k-1)T, \, X_{(k-1)T} \in \cov_{\beta}^{c} ] \\
&\leq  \P[\tau_{p,S^{c}} > kT | \tau_{p,S^{c}} > (k-1)T, \, X_{(k-1)T} \in \cov_{\beta}] \P[\tau_{p,S^{c}} > (k-1)T] \\
& \qquad +  \P[\tau_{p,S^{c}} > (k-1)T, \, X_{(k-1)T} \in \cov_{\beta}^{c}] \\
& \leq (1 - \Phi_{\beta}(S)(1 - 2 \beta^{-2})) \P[\tau_{p,S^{c}} > (k-1)T] +  \P[\tau_{p,S^{c}} > (k-1)T, \, X_{(k-1)T} \in \cov_{\beta}^{c}] \\
& \leq (1 - \Phi_{\beta}(S)(1 - 2 \beta^{-2}))\P[\tau_{p,S^{c}} > (k-1)T] +  k\Phi_{\beta}(S)^{4}, \\
\ee 
where Inequality \eqref{IneqMeta1Hit1} is used in the second-last line and Inequality \eqref{IneqMeta1Hit2} is used in the last line. Iterating and collecting terms, this gives
\be \label{IneqMeta1Hit3}
\P[\tau_{p,S^{c}} > kT] \leq (1 - \Phi_{\beta}(S)(1 - 2 \beta^{-2}))^{k} + k^{2} \Phi_{\beta}(S)^{4}.
\ee 

Fix any $\epsilon > 0$ and take $k = \lceil \beta \Phi_{\beta}(S)^{-1} \rceil$. By Part \textbf{(1)} of Assumption \ref{AssumptionsMeta1}, $k T \leq \Phi_{\beta}(S)^{1+\epsilon}$ for all $\beta > \beta_{0}(\epsilon)$ sufficiently large. Thus,  we can use Inequality \eqref{IneqMeta1Hit3} to conclude
\be 
\P[ \frac{\log(\tau_{x,S^{c}})}{\log(\Phi_{\beta}(S))} > 1 + \epsilon] &= \P[\tau_{x,S^{c}} > \Phi_{\beta}(S)^{1 + \epsilon}] \\
& \leq \P[\tau_{x,S^{c}} > k T]  \\
&\leq (1 - \Phi_{\beta}(S)(1 - 2 \beta^{-2}))^{k} + k^{2} \Phi_{\beta}(S)^{4} = o(1).
\ee 

This completes the proof of the lemma. 

\end{proof}

\subsection{Metastability and Spectral Gaps}

If one can partition the state space of a Markov chain into a collection of sets $S^{(1)},\ldots,S^{(k)}$ satisfying Assumption \ref{AssumptionsMeta1}, one typically expects the spectral gap of the Markov chain to be entirely determined by the typical transition rates between these sets. However, we must rule out a few possible sources of bad behaviour:

\begin{enumerate}
\item Very slow mixing in the ``tails" of the distribution could have an impact on the spectral gap.
\item A typical transition from one mode could land far out in the tails of the mode being entered, causing the walk to get ``stuck."
\item The transitions between modes might exhibit near-periodic behaviour, even if the Markov chain is not exactly periodic.
\item There might be metastability among \textit{collections} of modes. For example, there might be some $I \subset \{1,2,\ldots,k\}$ for which $\Phi_{\beta}(\cup_{i \in I} S^{(i)})$ is much smaller than $\min_{1 \leq i \leq k} \Phi_{\beta}(S^{(i)})$.
\end{enumerate}

Although detailed discussion of metastability is beyond the scope of the present paper, the first three types of behaviour can all cause the spectral gap to be very different from the prediction given by our metastability heuristic. The fourth behaviour simply says that you have chosen the ``wrong" partition of the state space, and that you should check the conditions again after joining several pieces of the partition together.

The following assumptions rule out these new complications:

\begin{assumptions} \label{AssumptionsMeta2}
Let $\Omega = \sqcup_{i=1}^{k} S^{(i)}$ be a partition of $\Omega$ into $k$ pieces. Set $\Phi_{\min} = \min(\Phi_{\beta}(S^{(1)}),\ldots, \Phi_{\beta}(S^{(k)}))$ and $\Phi_{\max} = \max(\Phi_{\beta}(S^{(1)}), \ldots, \Phi_{\beta}(S^{(k)}))$. Assume that:

\begin{enumerate}
\item \textbf{Metastability of Sets:} Each set $S^{(i)}$ satisfies Assumption \ref{AssumptionsMeta1} (with $\Phi_{\beta}(S)$ replaced by $\Phi_{\max}$ in Part \textbf{(1)} and replaced by $\Phi_{\min}$ in Parts \textbf{(2-4)}). We use the superscript $(i)$ to extend the notation of that assumption in the obvious way.
\item \textbf{Lyapunov Control of Tails:} 
Denote by $B_{r}(x)$ the ball of radius $r > 0$ around a point $x \in \Omega$. Assume there exist $0 < m, M < \infty$ satisfying
\be \label{IneqLyapContain}
\cup_{i=1}^{k} \cov_{\beta}^{(i)} \subset B_{M}(0), \qquad
B_{m}(0) \subset \cup_{i=1}^{k} \gd_{\beta}^{(i)}.
\ee 
Assume there exist a collection of privileged points $s_{i} \in \gd_{\beta}^{(i)}$ such that the function $V_{\beta}(x) = e^{\beta \min_{1 \leq i \leq k} \| x - s_{i} \|}$ satisfies
\be \label{IneqLyapMain}
(Q_{\beta} V_{\beta})( x) \leq (1 - \frac{1}{r_{3}(\beta)}) V_{\beta}( x) + r_{4} e^{\ell \beta}
\ee 
for all $x \in \Omega$, where $r_{3}, r_{4}$ are bounded by polynomials and $0 \leq \ell < m$. 
\item \textbf{Never Hitting $\cov_{\beta}^{c}$:} We have the following variant of Inequality \eqref{IneqNeverHittingBad1}:
\be \label{IneqNeverHittingBad2}
\sup_{x \in \cup_{i=1}^{k} \gd_{\beta}^{(i)}}  \P[\tau_{x, (\cup_{i=1}^{k}\cov_{\beta}^{(i)})^{c}}  < \Phi_{\min}^{-2}] &\leq \Phi_{\min}^{4}.
\ee 

\item \textbf{Non-Periodicity:} For all $1 \leq i \neq j \leq k$, 
\be \label{IneqNP1}
\inf_{\beta} \inf_{x \in S^{(i)} } Q_{\beta}(x,S^{(i)}) \equiv c_{2}^{(i)} > 0, 
\ee 
and,
\be \label{IneqNP2}
\sup_{x \in S^{(i)}} Q_{\beta}(x,S^{(j)} \backslash \gd_{\beta}^{(j)}) < \Phi_{\min}^{4}.
\ee 

\item \textbf{Connectedness:} There exists some $r_{5}$ bounded by a polynomial so that the graph with vertex set $\{1,2,\ldots,k\}$ and edge set
\be \label{IneqStrongCon}
\{ (i,j) \, : \,\min ( \inf_{x \in S^{(i)}} \P[X_{\tau_{x,(S^{(i)})^{c}}} \in S^{(j)}],   \inf_{x \in S^{(j)}} \P[X_{\tau_{x,(S^{(j)})^{c}}} \in S^{(i)}]) \geq r_{5}(\beta) \}
\ee 
is connected.
\end{enumerate} 
\end{assumptions}

 We then have:

\begin{lemma} [Spectral Gap and Conductance] \label{LemmaMeta2}
Let Assumptions \ref{AssumptionsMeta2} hold. Denote by $\lambda_{\beta}$ and $\Phi_{\beta}$ the spectral gap and conductance of $Q_{\beta}$. Then
\be \label{EqSpecGapCondConc}
\lim_{\beta \rightarrow \infty} \frac{\log(\lambda_{\beta})}{\log(\Phi_{\min})} = \lim_{\beta \rightarrow \infty} \frac{\log(\Phi_{\beta})}{\log(\Phi_{\min})} = 1.
\ee 

\end{lemma}

\begin{proof}

Define the candidate ``small set" 
\be 
R = \cup_{i=1}^{k} \gd_{\beta}^{(i)}.
\ee 

For convenience, we define $T_{\max} = T_{\max}(\beta) \equiv \Phi_{\min}^{-1.5}$ to be the longest time-scale of interest in this problem; we note that any mixing behaviour should occur on this time scale, while on the other hand there should be no entrances to the ``bad" set 
\be 
W \equiv (\cup_{i=1}^{k} \cov_{\beta}^{(i)})^{c}.
\ee  
In order to reduce notational clutter, we will frequently use $q$ with a subscript to refer to a function that is bounded by a polynomial and whose specific values are not of interest. 

We will begin by estimating the mixing rate of $Q_{\beta}$ for Markov chains started at points $x,y \in R$. We do this by coupling Markov chains $\{X_{t}\}_{t =0}^{T_{\max}}$, $\{Y_{t}\}_{t =0}^{T_{\max}}$ started at $X_{0} = x$, $Y_{0} = y$ for some $x,y \in R$ and trying to force them to collide.\footnote{Note that the Markov chains are only defined up until this ``maximal time" $T_{\max}$. This saves us from having to explicitly write $\min(\cdot, T_{\max})$, or adding extremely small terms that correspond to the  probability that various times exceed $T_{\max}$, error bounds in essentially all of the following calculations. This choice has virtually no other impact.} Roughly speaking, we will make the following two calculations:

\begin{enumerate}
\item If we run the two chains independently, the time it takes for them to both be in $\gd_{\beta}^{(i)}$ for the \textit{same $i$ simultaneously}, is not too much larger than the conjectured relaxation time $\Phi_{\min}^{-1}$. 
\item If we run two chains started on the same good set $\gd_{\beta}^{(i)}$, the two chains will couple long before either one  transitions from $\gd_{\beta}^{(i)}$ to another mode.
\end{enumerate}

We now give some further details, following this sketch. Let $x,y \in R$.

\textbf{Part 1: Time to be in same good set simultaneously.} We will run the chains independently until the first time
\be 
\psi_{1} = \inf \{t \geq 0 \, : \, \exists \, 1 \leq i \leq k \, \text{ s.t. } X_{t}, Y_{t} \in \gd_{\beta}^{(i)} \}
\ee 
that they are both in the same ``good" part of the partition. Define
\be
\psi_{2} = \inf \{t \geq 0 \, : \, \exists \, 1 \leq i \leq k \, \text{ s.t. } X_{t}, Y_{t} \in  S^{(i)} \},
\ee
the first time that $\{X_{t}\}$, $\{Y_{t}\}$ are both in the same part of our partition. For convenience, set $c_{2} = \min(c_{2}^{(1)}, \ldots, c_{2}^{(k)}, 0.5)$. By Inequalities \eqref{IneqMeta1Hit3}, \eqref{IneqNP1}, and \eqref{IneqStrongCon} 

\be \label{IneqMeta2InitHitting1}
\P[\psi_{2} \leq q_{1}(\beta) \Phi_{\min}^{-1}] \geq \frac{1}{2} c_{2} r_{5}(\beta)^{k}
\ee 
for some function $q_{1}$ that is bounded by a polynomial. By Inequality \eqref{IneqNP2}, 
\be 
\P[\psi_{1} = \psi_{2}] \geq 1-  T_{\max} \Phi_{\min}^{4} \geq 1 -  \Phi_{\min}^{2}.
\ee 
Combining this with Inequality \eqref{IneqMeta2InitHitting1},
\be \label{IneqMeta2InitHitting3}
\P[\psi_{1} \leq q_{1}(\beta) \Phi_{\min}^{-1}] \geq \frac{1}{q_{2}(\beta)}
\ee 
for some function $q_{2}$ that is bounded by a polynomial.

\textbf{Part 2: Mixing from same good set.}  If $\psi_{1} \geq T_{\max} - r_{1}(\beta)$, continue to evolve $\{X_{t}\}_{t \geq 0}$, $\{Y_{t}\}_{t \geq 0}$ independently. Otherwise, let $1 \leq i \leq k$ satisfy $X_{\psi_{1}}, \, Y_{\psi_{1}} \in S^{(i)}$. We then let $\{ \hat{X}_{t} \}_{t \geq 0}$, $\{ \hat{Y}_{t} \}_{t \geq 0}$ be Markov chains evolving according to the Metropolis-Hastings kernel with proposal distribution $Q_{\beta}$ and target distribution $\pi_{\beta} |_{S^{(i)}}$. We give these chains  initial points $\hat{X}_{0} = X_{\psi_{1}}$, $\hat{Y}_{0} = Y_{\psi_{1}}$ and couple them according to a maximal $r_{1}(\beta)$-step coupling (that is, a coupling that maximizes $\P[\hat{X}_{r_{1}(\beta)} = \hat{Y}_{r_{1}(\beta)}]$; such a coupling is known to exist \cite{griffeath1975maximal}).

We next observe that the following informal algorithm gives a valid coupling of the Markov chains $\{X_{t}\}_{t = \psi_{1}}^{\psi_{1} + r_{1}(\beta)}$, $\{\hat{X}_{t}\}_{t = 0}^{r_{1}(\beta)}$:
\begin{enumerate}
\item Run the full Markov chain $\{X_{t}\}_{t = \psi_{1}}^{\psi_{1} + r_{1}(\beta)}$ according to $Q_{\beta}$.
\item For all 
\be 
t < \tau_{\mathrm{bad}} \equiv \inf \{s \, : \, X_{\psi_{1} + s} \notin S^{(i)} \},
\ee  set $\hat{X}_{t} = X_{\psi_{1} + t}$.
\item If $ \tau_{\mathrm{bad}} < r_{1}(\beta)$, continue to evolve $\{\hat{X}_{s}\}_{s=\tau_{\mathrm{bad}}}^{r_{1}(\beta)}$ independently of $\{X_{t}\}_{t=0}^{r_{1}(\beta)}$.\footnote{Note that the particular choice made in this third step will not influence the analysis - we could make any measurable choice here.}
\end{enumerate}
We couple the pair of chains $\{X_{t}\}_{t = \psi_{1}}^{\psi_{1} + r_{1}(\beta)}$, $\{\hat{X}_{t}\}_{t = 0}^{r_{1}(\beta)}$ this way, and we couple $\{Y_{t}\}_{t = \psi_{1}}^{\psi_{1} + r_{1}(\beta)}$, $\{\hat{Y}_{t}\}_{t = 0}^{r_{1}(\beta)}$ analogously. Under these couplings, we have:  
\be 
\P[X_{\psi_{1} + r_{1}(\beta)} \neq Y_{\psi_{1} + r_{1}(\beta)}] &\leq \P[\hat{X}_{r_{1}(\beta)} \neq \hat{Y}_{r_{1}(\beta)}] + \P[X_{\psi_{1} + r_{1}(\beta)} \neq \hat{X}_{r_{1}(\beta)}] + \P[Y_{\psi_{1} + r_{1}(\beta)} \neq \hat{Y}_{r_{1}(\beta)}] \\
&\leq \beta^{-2} \Phi_{\max}  + 2 \Phi_{\min}^{4}, 
\ee 
where the first term is bounded by Inequality \eqref{IneqMeta1MixingRestrictions} and the last two terms are bounded by Inequality \eqref{IneqNeverHittingBad2}. 

Combining this with Inequality \eqref{IneqMeta2InitHitting3}, we conclude 
\be \label{IneqMeta2Minor1}
\P[X_{T_{1}} = Y_{T_{1}}] = (1 -o(1)),
\ee 
where $T_{1} = \lceil q_{1}(\beta) \Phi_{\min}^{-1} + r_{1}(\beta) \rceil$. 

This completes the proof of our two-stage analysis, as Inequality \eqref{IneqMeta2Minor1} gives a useful minorization bound for $x,y \in R$. 
Note that Inequality \eqref{IneqMeta2Minor1} is very close to a minorization condition in the sense of \cite{rosenthal1995minorization} for the small set $R$. Applying the closely related Lemma A.11 of \cite{mangoubi2017rapidp1} (on ArXiv, \cite{mangoubi2017concave}), the minorization bound \eqref{IneqMeta2Minor1} and the drift bound in Part \textbf{(2)} of Assumptions \ref{AssumptionsMeta2}, we find 
\be 
\| Q_{\beta}^{t}(x,\cdot) - \pi_{\beta}(\cdot) \|_{\mathrm{TV}} \leq M(\beta,x) e^{-\frac{t \Phi_{\min}}{q_{4}(\beta)}},
\ee 
where $q_{4}(\cdot) \geq 1$ and, for each $x$, $M(\cdot,x)$, are bounded by a polynomial.

By Theorem 2.1 of \cite{roberts1997geometric}, this implies

\be \label{IneqPrev21}
\lambda_{\beta} \geq  \frac{\Phi_{\min}}{q_{4}(\beta)}.
\ee

By Equation \eqref{IneqCheegPoin}, the conductance $\Phi_{\beta}$ of $Q_{\beta}$ satisfies

\be \label{IneqPrev22}
\lambda_{\beta} \leq 2 \Phi_{\beta} \leq 2 \Phi_{\min}.
\ee

Combining Inequalities \eqref{IneqPrev21} and \eqref{IneqPrev22}, we conclude 

\be 
 \frac{\Phi_{\min}}{q_{4}(\beta)} \leq  \frac{\Phi_{\beta}}{q_{4}(\beta)}, \, \lambda_{\beta} \leq 2 \Phi_{\min} \leq  2 \Phi_{\beta}.
\ee
 This immediately implies the limit in Equation \eqref{EqSpecGapCondConc}, completing the proof of the theorem.

\section{Application to Mixtures of Gaussians} \label{SecAppl}

We define the usual random-walk Metropolis algorithm:

\begin{defn} [Random Walk Metropolis Algorithm] \label{DefRWMAlg}
The transition kernel $K$ of the \textit{random walk Metropolis algorithm} with step size $\sigma > 0$ and target distribution $\pi$ on $\mathbb{R}^{d}$ with density $\rho$ is given by the following algorithm for sampling $X \sim K(x,\cdot)$:
\begin{enumerate}
\item Sample $\epsilon_{1} \sim \mathcal{N}(0,\sigma^{2})$ and $U_{1} \sim \mathrm{Unif}[0,1]$.
\item If 
\be 
U < \frac{\rho(x + \epsilon_{1})}{\rho(x)},
\ee 
set $X = x + \epsilon_{1}$. Otherwise, set $X=x$.
\end{enumerate}

\end{defn}

For $\sigma > 0$, define the mixture distribution 
\be  \label{EqDefMixDistBasic}
\pi_{\sigma} = \frac{1}{2} \mathcal{N}(-1,\sigma^{2}) + \frac{1}{2} \mathcal{N}(1,\sigma^{2})
\ee  
and denote its density by $f_{\sigma}$. Let $K_{\sigma}$ be the kernel from Definition \ref{DefRWMAlg} with step size $\sigma$ and target distribution $\pi_{\sigma}$.

Denote by $\lambda_{\sigma}$ the relaxation time of $K_{\sigma}$ (the reciprocal of the spectral gap of $K_\sigma$), and denote by $\Phi_{\sigma} = \Phi(K_{\sigma}, (-\infty,0))$ the Cheeger constant associated with kernel $K_{\sigma}$ and set $(-\infty,0)$.

We will state our two main results about this walk; the proofs are deferred until both results have been stated. First, we have an asymptotic formula for the Cheeger constant:

\begin{thm} \label{ThmRwmMultimodal1}
The Cheeger constant $\Phi_{\sigma}$ satisfies 
\be \label{EqMultiCheegAsymRwm}
\lim_{\sigma \rightarrow 0} (-2 \sigma^{2}) \, \log(\Phi_{\sigma}) = 1.
\ee 
\end{thm}

For fixed $x \in (-\infty,0)$, let $\{X_{t}^{(\sigma)}\}_{t \in \mathbb{N}}$ be a Markov chain with transition kernel $K_{\sigma}$ and initial point $X_{1}^{(\sigma)} = x$. Define the hitting time 
\be \label{EqDefTauSigmaX}
\tau_{x}^{(\sigma)} = \inf \{ t > 0 \, : \, X_{t}^{(\sigma)} \notin (-\infty,0)\}.
\ee

We also have the following estimate of the spectral gap and the hitting time:

\begin{thm} \label{ThmRwmMultimodal2}
For all $\epsilon > 0$ and fixed $x \in (-\infty,0)$, the hitting time $\tau_{x}^{(\sigma)}$ satisfies 
\be \label{EqMultiCheegHittingRwm}
\lim_{\sigma \rightarrow 0} \P[\frac{\log(\tau_{x}^{(\sigma)})}{\log(\Phi_{\sigma})} < 1 + \epsilon] = 1
\ee 
and the relaxation time satisfies 
\be  \label{IneqRelMultiRwm}
\lim_{\sigma \rightarrow 0} \frac{\log(\lambda_{\sigma})}{\log(\Phi_{\sigma})} = \lim_{\sigma \rightarrow 0} \frac{\log(\lambda_{\sigma})}{\log(\Phi(K_{\sigma}))} = 1.
\ee 

\end{thm}

\begin{remark}
This result implies that the Cheeger constant $\Phi(K_{\sigma})$ of $K_{\sigma}$ is close to the bottleneck ratio $\Phi_{\sigma} = \Phi(K_{\sigma}, (-\infty,0))$ associated with the set $(-\infty,0)$, at least for $\sigma$ very small. The set $(-\infty,0)$ is of course a natural guess for the set with the ``worst" conductance, though we do not know of any simple argument that would actually prove this. In some sense this is the motivation for the approach taken in this paper: it can be very hard to guess a good partition, even in a very simple example!
\end{remark}

We begin by proving Theorem \ref{ThmRwmMultimodal1}:

\begin{proof}[Proof of Theorem \ref{ThmRwmMultimodal1}]

Let  $\{X_{t}\}_{t \geq 0}$ be a Markov chain with transition kernel $K_{\sigma}$ and started at  $X_{0} \sim \pi_{\sigma}$ drawn according to the stationary distribution. Denote by $\phi_{\sigma}$  the density of the Gaussian with variance $\sigma^{2}$. Defining the set $\mathcal{E} = \{X_{0} < -\sigma^{-1} \} \cup \{|X_{1} - X_{0}| > \sigma^{-1} \}$, we have

\be 
\P[\{X_{0} < 0\} \cap \{\, X_{1} > 0\} \cap \mathcal{E}^{c}] &\leq \int_{-\sigma^{-1}}^{0} \int_{0}^{\sigma^{-1}} f_{\sigma}(x) \phi_{\sigma}(y-x) dx dy \\
& \leq 2\int_{-\sigma^{-1}}^{0} \int_{0}^{\sigma^{-1}}  \phi_{\sigma}(1+x) \phi_{\sigma}(y-x) dx dy \\
&= \frac{2}{\pi \sigma^{2}}\int_{-\sigma^{-1}}^{0} \int_{0}^{\sigma^{-1}}  e^{-\frac{1}{2 \sigma^{2}}( (1+x)^{2} + (y-x)^{2})} dx dy \\
&\leq \frac{2}{\pi \sigma^{2}} \int_{-\sigma^{-1}}^{0} \int_{0}^{\sigma^{-1}}  e^{-\frac{1}{2 \sigma^{2}}} = \frac{2}{\pi \sigma^{4}} e^{-\frac{1}{2 \sigma^{2}}}.
\ee 

We also have the simple bound
\be
\P[\mathcal{E}] &\leq \P[X_{0} < - \sigma^{-1}] + \P[|X_{1} - X_{0} | > \sigma^{-1}] \\
&\leq 2 \int_{-\infty}^{-\sigma^{-1}} \phi_{\sigma}(x) dx + 2 \int_{\sigma^{-1}}^{\infty} \phi_{\sigma}(x) dx \\
&\leq \frac{4}{\sqrt{2 \pi} \sigma} e^{-\frac{(\sigma^{-1}-1)^{2}}{2 \sigma^{2}}} \leq \frac{4}{\sqrt{2 \pi} \sigma} e^{-\frac{1}{3} \sigma^{-3}},
\ee
where the last inequality holds for all $\sigma$ sufficiently small. Putting these two bounds together, we have for all $\sigma > 0$ sufficiently small that 
\be 
\P[X_{0} < 0, \, X_{1} > 0] \leq \frac{1}{\pi \sigma^{4}} e^{-\frac{1}{2 \sigma^{2}}} + \frac{4}{\sqrt{2 \pi} \sigma} e^{-\frac{1}{3} \sigma^{-3}}.
\ee 
Taking logs, this immediately proves
\be 
\lim_{\sigma \rightarrow 0} (-2 \sigma^{2}) \, \log(\Phi_{\sigma}) \leq 1,
\ee 
 the desired upper bound on the left-hand side of  Inequality \eqref{EqMultiCheegAsymRwm}. To prove the lower bound on this quantity, begin by defining the intervals $I_{\sigma} = (-2 \sigma^{20}, - \sigma^{20})$ and $J_{\sigma} = (\sigma^{10},2 \sigma^{10})$. Since $\sigma^{20} \ll \sigma^{10}$ for $\sigma$ small, we have for sufficiently small $\sigma > 0$ 
\be 
\inf_{x \in I_{\sigma}, y \in J_{\sigma}} \frac{f_{\sigma}(y)}{f_{\sigma}(x)} \geq 1. \ee 
Informally, this means that any proposed step from $I_{\sigma}$ to $J_{\sigma}$ will be accepted. Thus, letting $Y \sim \mathcal{N}(0,\sigma^{2})$ be independent of $X_{1}$, we have
\be 
\P[X_{0} < 0, X_{1} > 0] &\geq \P[ X_{0} \in I_{\sigma}, X_{1} \in J_{\sigma}] \\
&\geq \P[X_{0} \in I_{\sigma}] \inf_{x \in I_{\sigma}} \P[ X_{0} + Y \in J_{\sigma} \, | \, X_{0} = x] \\
&\geq \left(\frac{\sigma^{20}}{4 \sqrt{2 \pi}} e^{-\frac{1}{2 \sigma^{2}}} \right) \times \left( \frac{\sigma^{10}}{4 \sqrt{2 \pi}} \right),
\ee 
where the last inequality holds for all $\sigma > 0$ sufficiently small. Taking logs, this proves
\be 
\lim_{\sigma \rightarrow 0} (-2 \sigma^{2}) \, \log(\Phi_{\sigma}) \geq 1,
\ee 
completing the proof of Inequality \eqref{EqMultiCheegAsymRwm}.
\end{proof}

\begin{proof}[Proof of Theorem \ref{ThmRwmMultimodal2}]
We defer some of the longer exact calculations in this proof to Appendix \ref{SecAppendix}, retaining here the key steps that might be used to prove similar metastability results for other Markov chains. To prove Theorem \ref{ThmRwmMultimodal2}, it is enough to verify Assumptions \ref{AssumptionsMeta1} and \ref{AssumptionsMeta2} for the sets 
\be 
S^{(1)} = (-\infty,0), \qquad S^{(2)} = [0,\infty),
\ee 
with decomposition of $S^{(1)}$
\be 
\gd_{\beta}^{(1)} = (-\sigma^{-9},0), \quad \cov_{\beta}^{(1)} = (-\sigma^{-10},0), \quad \bd_{\beta}^{(1)} = (-\infty, -\sigma^{-10})
\ee 
and $\gd_{\beta}^{(2)}$, $\cov_{\beta}^{(2)}$, $\bd_{\beta}^{(2)}$ defined analogously (see Figure \ref{FigPartitionEx}). Note that Part \textbf{(1)} of Assumptions \ref{AssumptionsMeta1} follows immediately from Inequality \eqref{EqMultiCheegAsymRwm}, which we have already proved. 

\begin{figure}[H]
\includegraphics[scale=0.5]{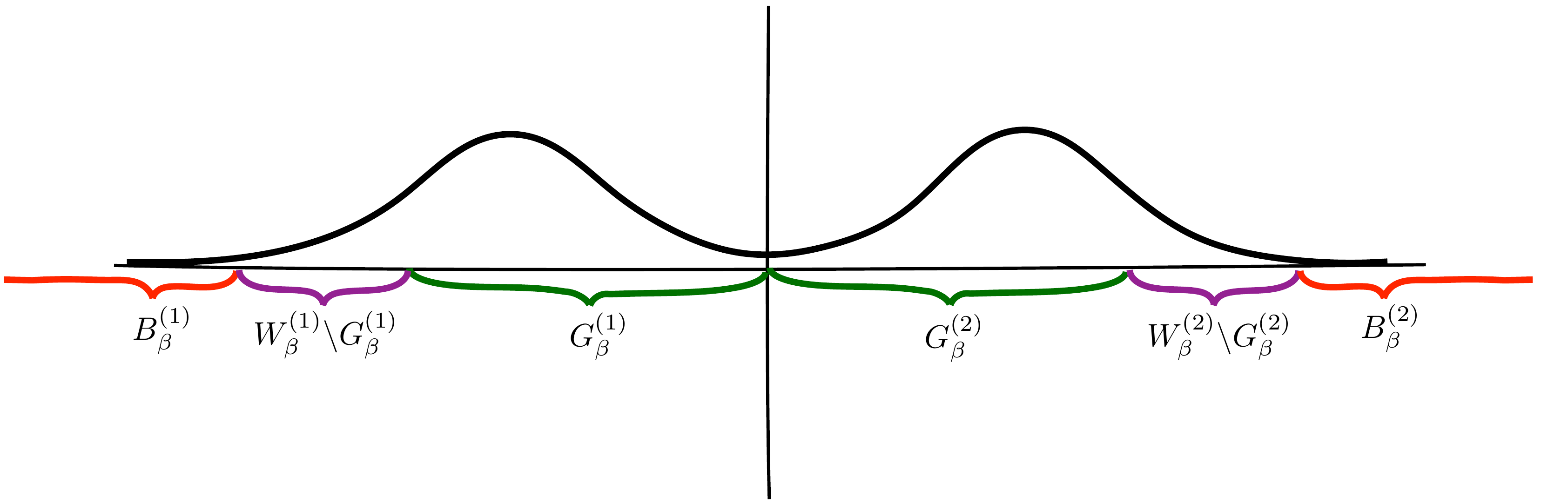}
\caption{A cartoon plot of the target density $\mu_{\sigma}$ with the regions illustrated. Note that we have substantially compressed the regions; in a to-scale drawing, $\bd_{\beta}$ would not be visible. \label{FigPartitionEx}}
\end{figure}

Denote by $\hat{K}_{\sigma}$ the Metropolis-Hastings transition kernel on $(-\infty,0)$ that has as its proposal kernel $K_{\sigma}$ and as its target distribution the density 
\be  
\hat{\rho}_{\sigma}(x) = 2 f_{\sigma}(x), \quad x \in (-\infty,0).
\ee

We begin by proving some stronger Lyapunov-like bounds for $K_{\sigma}$ and $\hat{K}_{\sigma}$:

\begin{lemma} \label{LemmaLyapunovEx}
Let $V_{\sigma}(x) = e^{\sigma^{-1} \min(\|x - 1\|, \|x+1\|)}$. Then there exist $0 < \alpha \leq 1$, $0 \leq M,C < \infty$ so that for all $K \in \{K_{\sigma}, \hat{K}_{\sigma}\}$ and $x \in (-\infty,-M \sigma)$,
\be \label{IneqLyapMultiMH}
(K V_{\sigma})(x) \leq (1 - \alpha) V_{\sigma}(x) + C.
\ee 
Furthermore, Part \textbf{(2)} of Assumptions \ref{AssumptionsMeta2} holds.
\end{lemma}

\begin{proof}
Proof is deferred to Appendix \ref{SecAppendix}. 
\end{proof}

We next check the main condition:

\begin{thm} \label{ThmMulti2AppMainMixingMode}
With notation as above, Part \textbf{(2)} of  Assumptions \ref{AssumptionsMeta1} is satisfied.
\end{thm}

\begin{proof}

We begin with a weak estimate of mixing from within a good set: 

\begin{lemma} \label{LemmaInitialModeMixingEstimateMH}
Fix $0 < \delta < \frac{1}{20}$. With notation as above, there exist some constants $0 < a_{1}, A_{1} < \infty$ so that 
\be  \label{IneqMixingBulkHatMH}
\sup_{-\sigma^{-11} < x,y < -\delta} \| \hat{K}_{\sigma}^{T}(x,\cdot) -  \hat{K}_{\sigma}^{T}(y,\cdot) \|_{\mathrm{TV}} \leq A_{1} e^{-a_{1} \sigma^{-1}}
\ee
for $T > A_{1} \sigma^{-a_{1}}$.
\end{lemma}

\begin{proof}
Proof is deferred to Appendix \ref{SecAppendix}. 
\end{proof}

Note that this bound is not good enough for our conclusions, since our upper bound $e^{-a_{1} \sigma^{-1}}$ is still very large compared to the conductance of interest. We improve the bound by iterating it several times:

\begin{lemma} \label{LemmaIterateHatMixingMH}
Fix $0 < \delta < \frac{1}{20}$. There exist constants $0 < a_{2},A_{2} < \infty$ depending only on $\delta$ so that 
\be  \label{EqVeryStrongHatMixing2MH}
\sup_{-\sigma^{-9} < x,y < -\delta} \| \hat{K}_{\sigma}^{S}(x,\cdot) -  \hat{K}_{\sigma}^{S}(y,\cdot) \|_{\mathrm{TV}} \leq A_{2} e^{-a_{2} \sigma^{-5}}
\ee
for $S = \lceil A_{2} \sigma^{-a_{2}} \rceil$. Furthermore, there exist  constants $0 < a_{3}, A_{3} < \infty$ such that

\be \label{IneqEscapeSmallSet}
\sup_{x \in (-\delta,0)} \P[ \tau_{x,(-\delta,0)^{c}} < A_{3} \sigma^{-a_{3}}] \geq 1- e^{-\sigma^{-10}}.
\ee 

\end{lemma}

\begin{proof}
Proof is deferred to Appendix \ref{SecAppendix}. 
\end{proof}

Fix $\delta = 0.01$. Combining Inequality \eqref{EqVeryStrongHatMixing2MH} with the bound \eqref{IneqEscapeSmallSet} on the length of excursions above $- \delta$ completes the proof of the theorem.

\end{proof}

\begin{lemma} 
With notation as above, Parts \textbf{(3-4)} of Assumptions \ref{AssumptionsMeta1} and Part \textbf{(3)} of Assumptions \ref{AssumptionsMeta2} hold.
\end{lemma}

\begin{proof}
These all follow immediately from Lemma \ref{LemmaLyapunovEx} and the definition of our partition.
\end{proof}

Next, note that Part \textbf{(1)} of  Assumptions \ref{AssumptionsMeta2} holds by the symmetry of $S^{(1)}, S^{(2)}$ and the fact that we have already checked Assumptions \ref{AssumptionsMeta1}.

Thus, it remains only to check Part \textbf{(4)} of Assumptions \ref{AssumptionsMeta2}:

\begin{lemma}
With notation as above, Part \textbf{(4)} of Assumptions \ref{AssumptionsMeta2} holds.
\end{lemma} 

\begin{proof}
It is immediately clear that $K_{\sigma}(x,(-\infty,x]) \geq \frac{1}{2}$ for all $x \in \mathbb{R}$, which implies Inequality \eqref{IneqNP1}. Standard Gaussian inequalities imply
\be 
\sup_{x < 0} K_{\sigma}(x,(\sigma^{10},\infty)) \leq e^{-\sigma^{9}}
\ee 
for $\sigma < \sigma_{0}$ sufficiently small. Combining this with Inequality \eqref{EqMultiCheegAsymRwm}  completes the proof of Inequality \eqref{IneqNP2}.
\end{proof}

Since we have verified all the assumptions of Lemma \ref{LemmaMeta1} and \ref{LemmaMeta2}, applying them completes the proof of Theorem \ref{ThmRwmMultimodal2}. 

\end{proof}

\end{proof}

\section*{Acknowledgement}
We would like to thank Neil Shephard and Gareth Roberts for some questions about multimodal distributions and helpful comments on initial results.

\bibliographystyle{alpha}
\bibliography{HMC}

\begin{thebibliography}{WSH09b}

\bibitem[BL15]{beltran2015martingale}
Johel Beltr{\'a}n and Claudio Landim.
\newblock A martingale approach to metastability.
\newblock {\em Probability Theory and Related Fields}, 161(1-2):267--307, 2015.

\bibitem[Che70]{cheeger1970lower}
Jeff Cheeger.
\newblock A lower bound for the smallest eigenvalue of the {L}aplacian.
\newblock {\em Problems in analysis}, pages 195--199, 1970.

\bibitem[Gri75]{griffeath1975maximal}
David Griffeath.
\newblock A maximal coupling for {M}arkov chains.
\newblock {\em Zeitschrift f{\"u}r Wahrscheinlichkeitstheorie und verwandte
  Gebiete}, 31(2):95--106, 1975.

\bibitem[JSTV04]{jerrum2004elementary}
Mark Jerrum, Jung-Bae Son, Prasad Tetali, and Eric Vigoda.
\newblock Elementary bounds on {P}oincar{\'e} and log-{S}obolev constants for
  decomposable {M}arkov chains.
\newblock {\em Annals of Applied Probability}, pages 1741--1765, 2004.

\bibitem[Lan18]{landim2018metastable}
Claudio Landim.
\newblock Metastable {M}arkov chains.
\newblock {\em arXiv preprint arXiv:1807.04144}, 2018.

\bibitem[LS88]{lawler1988bounds}
Gregory~F Lawler and Alan~D Sokal.
\newblock Bounds on the $l^2$ spectrum for {M}arkov chains and {M}arkov
  processes: a generalization of {C}heeger's inequality.
\newblock {\em Transactions of the American mathematical society},
  309(2):557--580, 1988.

\bibitem[MPS18]{mangoubi2018Cheeger}
Oren Mangoubi, Natesh~S Pillai, and Aaron Smith.
\newblock Does {H}amiltonian {M}onte {C}arlo mix faster than a random walk on
  multimodal densities?
\newblock {\em arXiv preprint arXiv:1808.03230}, 2018.

\bibitem[MR02]{madras2002markov}
Neal Madras and Dana Randall.
\newblock Markov chain decomposition for convergence rate analysis.
\newblock {\em Annals of Applied Probability}, pages 581--606, 2002.

\bibitem[MS17a]{mangoubi2017rapidp1}
Oren Mangoubi and Aaron Smith.
\newblock Mixing of {H}amiltonian {M}onte {C}arlo on strongly log-concave
  distributions 1: Continuous dynamics.
\newblock {\em preprint}, 2017.

\bibitem[MS17b]{mangoubi2017concave}
Oren Mangoubi and Aaron Smith.
\newblock Rapid mixing of {H}amiltonian {M}onte {C}arlo on strongly log-concave
  distributions.
\newblock {\em arXiv preprint arXiv:1708.07114}, 2017.

\bibitem[MT94]{meyn1994computable}
Sean~P Meyn and Robert~L Tweedie.
\newblock Computable bounds for geometric convergence rates of {M}arkov chains.
\newblock {\em The Annals of Applied Probability}, pages 981--1011, 1994.

\bibitem[MT96]{mengersen1996rates}
Kerrie~L Mengersen and Richard~L Tweedie.
\newblock Rates of convergence of the {H}astings and {M}etropolis algorithms.
\newblock {\em The Annals of Statistics}, 24(1):101--121, 1996.

\bibitem[OV05]{olivieri2005large}
Enzo Olivieri and Maria~Eul{\'a}lia Vares.
\newblock {\em Large deviations and metastability}, volume 100.
\newblock Cambridge University Press, 2005.

\bibitem[PS17]{pillai2017elementary}
Natesh~S Pillai and Aaron Smith.
\newblock Elementary bounds on mixing times for decomposable {M}arkov chains.
\newblock {\em Stochastic Processes and their Applications}, 127(9):3068--3109,
  2017.

\bibitem[Res13]{resnick2013probability}
Sidney~I Resnick.
\newblock {\em A probability path}.
\newblock Springer Science \& Business Media, 2013.

\bibitem[RGG97]{roberts1997weak}
Gareth~O Roberts, Andrew Gelman, and Walter~R Gilks.
\newblock Weak convergence and optimal scaling of random walk {M}etropolis
  algorithms.
\newblock {\em The Annals of Applied Probability}, 7(1):110--120, 1997.

\bibitem[Ros95]{rosenthal1995minorization}
Jeffrey~S Rosenthal.
\newblock Minorization conditions and convergence rates for {M}arkov chain
  {M}onte {C}arlo.
\newblock {\em Journal of the American Statistical Association},
  90(430):558--566, 1995.

\bibitem[RR97]{roberts1997geometric}
Gareth~O Roberts and Jeffrey~S Rosenthal.
\newblock Geometric ergodicity and hybrid {M}arkov chains.
\newblock {\em Electron. Comm. Probab}, 2(2):13--25, 1997.

\bibitem[WSH09a]{woodard2009sufficient}
Dawn Woodard, Scott Schmidler, and Mark Huber.
\newblock Sufficient conditions for torpid mixing of parallel and simulated
  tempering.
\newblock {\em Electronic Journal of Probability}, 14:780--804, 2009.

\bibitem[WSH09b]{woodard2009conditions}
Dawn~B Woodard, Scott~C Schmidler, and Mark Huber.
\newblock Conditions for rapid mixing of parallel and simulated tempering on
  multimodal distributions.
\newblock {\em The Annals of Applied Probability}, 19(2):617--640, 2009.

\end{thebibliography}

\appendix 

\section{Technical Bounds from the Proof of Theorem \ref{ThmRwmMultimodal2}} \label{SecAppendix}

We prove some technical lemmas that occur in the proof of Theorem \ref{ThmRwmMultimodal2}.

\begin{proof} [Proof of Lemma \ref{LemmaLyapunovEx}]
Let $a = a_{\sigma}$ be the unique local minimum of $f_{\sigma}$ in the interval $(-2, -0.5)$ (it is clear one such exists for all $\sigma > \sigma_{0}$ sufficiently large and that $a_{\sigma}$ is within distance $O(e^{-\frac{1}{3 \sigma^{2}}})$ of $-1$). Let $Q_{\sigma}$ be the transition kernel given in Definition \ref{DefRWMAlg} with step size $\sigma$ and target distribution $\mathcal{N}(a,4\sigma^{2})$. Let $L_{\sigma}(x) = e^{-\sigma^{-1} \|x-1\|}$. By a standard  computation (\textit{e.g.,} keeping track of the constants in the proof of Theorem 3.2 of \cite{mengersen1996rates}), there exist $0 < \alpha \leq 1$, $0 \leq C < \infty$ so that 
\be 
(Q_{\sigma} L_{\sigma})(x) \leq (1-\alpha)L(x) + C
\ee 
for all $x \in \mathbb{R}$ and $Q \in \{Q_{\sigma}, \hat{Q}_{\sigma}\}$. Next, observe that
\be 
\inf_{x \in (-\infty,-10 \sigma)} \frac{d^{2}}{dx^{2}} -\log(f_{\sigma}(x)) \geq \frac{1}{8 \sigma^{2}}.
\ee 
In particular, $f_{\sigma}$ is strongly log-concave on the interval $(-\infty,-10 \sigma)$ with the same parameter as the density of $\mathcal{N}(a,4\sigma^{2})$. Thus, if we fix $M > 0$ and $x \in (-\infty,-M \sigma)$ and let $X \sim K_{\sigma}(x,\cdot)$, we have for $K \in \{K_{\sigma}, \hat{K}_{\sigma}\}$
\be \label{LyapProofPt1}
(K V_{\sigma})(x) &\leq (Q_{\sigma} L_{\sigma})(x) + \E[V_{\sigma}(X) \textbf{1}_{X >  -10 \sigma}] \\
&\leq (1 -\alpha) L_{\sigma}(x) + C + \E[V_{\sigma}(X) \textbf{1}_{X > -10 \sigma}] \\ 
&= (1 -\alpha) V_{\sigma}(x) + C + \E[V_{\sigma}(X) \textbf{1}_{X >  -10 \sigma}].
\ee 
Let $Y \sim \mathcal{N}(0,\sigma^{2})$. As $M \rightarrow \infty$, for $K = \hat{K}_{\sigma}$ we can then bound the last term by
\be \label{LyapProofPt2}
\E[V_{\sigma}(X) \textbf{1}_{X >  -10 \sigma}] \leq V_{\sigma}(x) \E[e^{\sigma^{-1}Y} \textbf{1}_{x+ Y \in [-10\sigma,0] }] = (1+o(1)) V_{\sigma}(x).
\ee 
Combining Inequalities \eqref{LyapProofPt1} and \eqref{LyapProofPt2} completes the proof of Inequality \eqref{IneqLyapMultiMH} in the case $K = \hat{K}_{\sigma}$. In the case $K = K_{\sigma}$, we replace Inequality \eqref{LyapProofPt2} by the similar bound:
\be 
\E[V_{\sigma}(X) \textbf{1}_{X >  -10 \sigma}] \leq V_{\sigma}(x) \E[e^{\sigma^{-1}Y} \textbf{1}_{x+ Y \in [-10\sigma,10 \sigma] }] + V_{\sigma}(x) \P[x+Y > 10 \sigma] = (1+o(1)) V_{\sigma}(x)
\ee 
to obtain the same conclusion.

Finally, Part \textbf{(2)} of Assumptions \ref{AssumptionsMeta2} immediately follows from Inequality \eqref{IneqLyapMultiMH} in the case $K = K_{\sigma}$ and the trivial inequality 
\be 
\sup_{|x| \leq M \sigma} (K_{\sigma} V_{\sigma})(x) \leq e^{\sigma^{-3}}
\ee 
for any fixed $M$ and all $\sigma < \sigma_{0} = \sigma_{0}(A)$ sufficiently small.
\end{proof}

\begin{proof} [Proof of Lemma \ref{LemmaInitialModeMixingEstimateMH}]

Fix $\alpha, C$ as in Lemma \ref{LemmaLyapunovEx} and let $\mu$ be the uniform distribution on the interval $I = [-1-\frac{10 C}{\alpha} \sigma, -1 + \frac{10 C}{\alpha} \sigma]$. We note that $\hat{K}_{\sigma}$ inherits the following \textit{minorization condition} from the standard Gaussian:
\be \label{IneqSimpleMinGauss}
\inf_{x \in I} \inf_{J \subset I} \hat{K}_{\sigma}(x,J) \geq \epsilon \mu(J)
\ee 
for some $\epsilon > 0$ that does not depend on $\sigma$. 

Fix $-\sigma^{-10} < x,y < -\delta$. Applying the popular ``drift-and-minorization" bound in Section 10 of \cite{meyn1994computable}, using the ``drift" bound in Inequality \eqref{IneqLyapMultiMH} and the ``minorization" bound in Inequality \eqref{IneqSimpleMinGauss} gives a bound of the form:
\be \label{IneqProtoInitMixing}
\sup_{-\sigma^{-10} < x,y < -\delta}  \| \hat{K}_{\sigma}^{T}(x,\cdot) -  \hat{K}_{\sigma}^{T}(y,\cdot) \|_{\mathrm{TV}} \leq B_{1} e^{-b_{1} \sigma^{-1}} + 2 \sup_{-\sigma^{-10} < x < -\delta} \P[\tau_{x, (-M \sigma,\infty)} < t]
\ee
for all $T > B_{1} \sigma^{-b_{1}}$, where $0 < b_{1}, B_{1}$ are constants that do not depend on $\sigma$. Note that the second term on the right-hand side, which does not appear in \cite{meyn1994computable}, represents the possibility that a Markov chain ever escapes from the set $(-\infty,-M \sigma)$ on which the drift bound \eqref{IneqLyapMultiMH} holds.

Fix $-\sigma^{-10} < x < -\delta$ and let $\{X_{t}\}_{t \geq 0}$ be a Markov chain with transition kernel $\hat{K}_{\sigma}$ and starting point $X_{0} = x$. Let $\tau = \inf \{t \geq 0 \, : \, X_{t} > -M \sigma \}$. By \eqref{IneqLyapMultiMH}, we have for all $t \in \mathbb{N}$
\be 
\E[V_{\sigma}(X_{t}) \textbf{1}_{\tau \geq t-1}] \leq (1-\alpha)^{t} V_{\sigma}(X_{0}) + \frac{C}{\alpha}.
\ee 
Thus, by Markov's inequality,
\be 
\P[\tau \leq t] &\leq e^{-\sigma^{-1} (-M \sigma + 1)} \sum_{s=0}^{t} ((1-\alpha)^{t} V_{\sigma}(X_{0}) + \frac{C}{\alpha}) \\
&\leq  t \, e^{-\sigma^{-1} (-M \sigma + 1)} (e^{\sigma^{-1} (-\delta+1)} + \frac{C}{\alpha}).
\ee 
Combining this with Inequality \eqref{IneqProtoInitMixing} completes the proof of the lemma.
\end{proof}

\begin{proof}[Proof of Lemma \ref{LemmaIterateHatMixingMH}]

We denote by $\{X_{t}\}_{t \geq 0}$ a Markov chain with transition kernel $\hat{K}_{\sigma}$ and some starting point $X_{0} = x$. To improve on the bound in Lemma \ref{LemmaInitialModeMixingEstimateMH}, we must control what can occur when coupling does not happen quickly. There are two possibilities to control: the possibility that $\{X_{t}\}_{t \geq 0}$ goes above $- \delta$, and the possibility that it goes below $-\sigma^{-10}$. The latter is easier to control; by Inequality \eqref{IneqLyapMultiMH} and Markov's inequality, for all $\epsilon > 0$ there exist constants $c_{1} = c_{1}(\epsilon), C_{1} = C_{1}(\epsilon) > 0$ so that  

\be  \label{IneqVeryLowExcursionsVeryUnlikely}
\sup_{|X_{0}| \leq \sigma^{-\beta}} \P[\min_{1 \leq t \leq e^{\sigma^{-\beta}}} X_{t} < -\sigma^{-\beta - \epsilon}] &\leq e^{\sigma^{-\beta}} \, \sup_{|X_{1}| \leq \sigma^{-\beta}} \sup_{1 \leq t \leq e^{\sigma^{-\beta}}} \frac{\E[e^{|X_{t}|}]}{e^{\sigma^{-\beta - \epsilon}}} \\
&\leq e^{\sigma^{-\beta}} \left( \frac{e^{\sigma^{-\beta}} + \alpha^{-1} C}{e^{\sigma^{-\beta - \epsilon}}} \right) \\
&\leq C_{1} e^{-c_{1} \sigma^{-\beta - \epsilon}}
\ee
uniformly in $\beta \geq 1$.

The possibility that $\{X_{t} \}_{t \geq 0}$ goes above $- \delta$ cannot be controlled in the same way, because it does not have negligible probability on the time scale of interest. Instead, we use the fact that $X_{t}$ will generally exit the interval $(- \delta, 0)$ fairly quickly, often to the interval $(-\infty, - \delta)$. 

To see this, fix $x \in (- \delta, 0)$ and let $\{X_{t}\}_{t \geq 0}$ have starting point $X_{0} = x$. Next, let $\{\epsilon_{t}\}_{t \geq 1}$ be a sequence of i.i.d. $\mathcal{N}(0,\sigma^{2})$ random variables and let $Y_{t} = X_{0} + \sum_{s=1}^{t} \epsilon_{t}$. For $I \subset \mathbb{R}$, let 
\be 
\psi_{x,I} = \inf \{t \geq 0 \, : \, Y_{t} \in I\}
\ee 
be the hitting time of $I$ for the Markov chain $\{Y_{t}\}_{t \geq 0}$. Observing the forward mapping representation of $K_{\sigma}$ in Definition \ref{DefRWMAlg} and that $f_{\sigma}$ is monotone on $(- \delta, 0)$, it is clear that we can couple $\{X_{t}\}_{t \geq 0}$, $\{Y_{t}\}_{t \geq 0}$ so that 
\be  \label{IneqStochasticDomination}
X_{t} \leq Y_{t}, \quad \forall \, 0 \leq t \leq \min(\tau_{x,(-\delta,0)^{c}}, \psi_{x,(-\delta,0)^{c}}).
\ee 

But by standard calculations for simple random walk, \footnote{To see the first inequality, note that a direct calculation for Gaussians gives $\sup_{x \in (-\delta,0)} \P[\psi_{x,(-\delta,0)^{c}} > C_{2}' \sigma^{-c_{2}'}] > C_{2}'' > 0$ for some $C_{2}', c_{2}' C_{2}'' > 0$; applying the strong Markov property to iterate this bound as in the proof of Inequality \eqref{IneqMeta1Hit3} gives the desired conclusion. The second inequality follows from the observation that $\P[Y_{1} > C_{3}' \sigma] > C_{3}'' > 0$ for some constants $C_{3}', C_{3}'' > 0$ and then the well-known ``gambler's ruin" calculation (see \textit{e.g.} Section 10.14.4 of \cite{resnick2013probability}). } 
\be 
\sup_{x \in (-\delta,0)} \P[\psi_{x,(-\delta,0)^{c}} > C_{2} \sigma^{-c_{2}}] \leq \sigma^{2}, \quad \inf_{x \in (-\delta,0)} \P[Y_{\psi_{x,(-\delta,0)^{c}}} < -\delta] > C_{3} \sigma
\ee

for some constants $c_{2}, C_{2}, C_{3}$ that do not depend on $\sigma$. Combining this with Inequality \eqref{IneqStochasticDomination} and noting that $\{X_{t}\}_{t \geq 0}$ never exits $(-\infty,0)$ by construction, we find
\be  
\sup_{x \in (-\delta,0)} \P[ \tau_{x,(-\delta,0)^{c}} < C_{2} \sigma^{-c_{2}}] \geq C_{3} \sigma - \sigma^{2} = C_{3}(\sigma) (1 - o(1)).
\ee

Noting that these bounds are uniform over the starting point $X_{0} \in (-\delta,0)$, we find for $k \in \mathbb{N}$
\be 
\sup_{x \in (-\delta,0)} \P[ \tau_{x,(-\delta,0)^{c}} < k \,C_{2} \sigma^{-c_{2}}] \geq 1 - (1 - C_{3}(\sigma) - \sigma^{2})^{k}.
\ee 
Taking $k$ very large ($k > \sigma^{-12}$ suffices) gives

\be 
\sup_{x \in (-\delta,0)} \P[ \tau_{x,(-\delta,0)^{c}} < C_{4} \sigma^{-c_{4}}] \geq 1- e^{-\sigma^{-10}}
\ee 
for some constants $0 \leq c_{4}, C_{4} < \infty$, which is exactly Inequality \eqref{IneqEscapeSmallSet}

Combining the bound  \eqref{IneqMixingBulkHatMH} on the mixing of $\hat{K}_{\sigma}$ on $(-\sigma^{-11},-\delta)$ with the bound \eqref{IneqVeryLowExcursionsVeryUnlikely} on the possibility of excursions below $-\sigma^{-11}$ and the bound \eqref{IneqEscapeSmallSet} on the length of excursions above $- \delta$ completes the proof of the lemma.

\end{proof}

\end{document}